\documentclass[11pt]{article}
\usepackage{a4wide,amsfonts,amsmath,amssymb,amsthm,enumerate,color,hyperref,amsthm,epsfig,graphicx,bm,stmaryrd}

\usepackage[utf8]{inputenc}
\usepackage[T1]{fontenc}

\newcommand{\PP}{\mathcal P}

\renewcommand{\P}{\mathcal P}
\newcommand{\J}{\mathcal J}
\newcommand{\I}{\mathcal I}
\newcommand{\U}{\mathcal U}

\newcommand{\M}{{\mathcal W}}
\newcommand{\R}{\mathbb R}
\newcommand{\RR}{\mathbb R}
\newcommand{\HH}{\mathcal H}

 \newcommand{\varphim}{\varphi_{m,2}}
 \newcommand{\psim}{\psi_{m,2}}

\newcommand{\dt}{\frac{{\rm d}}{{\rm d}t}}

\renewcommand{\U}{\mathcal{U}}

\newcommand{\dom}{{\rm dom}\,}

\newcommand{\gradw}{{\rm grad}_\M}

\newcommand{\betat}{{\beta_t}}

\newtheorem{theorem}{Theorem}
\newtheorem{proposition}{Proposition}

\newtheorem{corollary}{Corollary}

\newtheorem{remark}{Remark}

\DeclareFontFamily{U}{txsyc}{}
\DeclareFontShape{U}{txsyc}{m}{n}{
   <-> txsyc%
}{}
\DeclareFontShape{U}{txsyc}{bx}{n}{
   <-> txbsyc%
}{}
\DeclareFontShape{U}{txsyc}{l}{n}{<->ssub * txsyc/m/n}{}
\DeclareFontShape{U}{txsyc}{b}{n}{<->ssub * txsyc/bx/n}{}
\DeclareSymbolFont{symbolsC}{U}{txsyc}{m}{n}
\SetSymbolFont{symbolsC}{bold}{U}{txsyc}{bx}{n}
\DeclareFontSubstitution{U}{txsyc}{m}{n}
\DeclareMathSymbol{\df}{\mathrel}{symbolsC}{"42}
\DeclareMathSymbol{\fd}{\mathrel}{symbolsC}{"43}
\DeclareMathSymbol{\lJoin}{\mathrel}{symbolsC}{"58}
\DeclareMathSymbol{\rJoin}{\mathrel}{symbolsC}{"59}

\newcommand{\f}[2]{\frac{#1}{#2}}

\newcommand{\cC}{\mathcal{C}}

\newcommand{\cW}{\mathcal{W}}

\newcommand{\NN}{\mathbb{N}}

\newcommand{\TT}{\mathbb{T}}
\newcommand{\ZZ}{\mathbb{Z}}

\newcommand{\iy}{\infty}

\newcommand{\lt}{\left}
\newcommand{\me}{\medskip}
\newcommand{\na}{\nabla}
\newcommand{\pa}{\partial}
\newcommand{\ri}{\rightarrow}
\newcommand{\rt}{\right}
\newcommand{\val}{{\rm val}}
\newcommand{\sm}{\smallskip}
\newcommand{\tr}{\triangle}

\newcommand{\app}{\mathrm{gap}}

\newcommand{\dive}{\mathrm{div}}

\newcommand{\fo}{\forall\ }

\newcommand{\lan}{\lt\langle}
\newcommand{\lve}{\lt\vert}

\newcommand{\osc}{\mathrm{osc}}

\newcommand{\ran}{\rt\rangle}
\newcommand{\rve}{\rt\vert}

\newcommand{\st}{\,:\,}

\newcommand{\un}{\mathrm{1}}

\newcommand{\ind}[1]{\mathrm{\un}_{\! #1}}

\newcommand{\bq}{\begin{eqnarray*}}
\newcommand{\bqn}[1]{\begin{eqnarray}\label{#1}}
\newcommand{\eq}{\end{eqnarray*}}
\newcommand{\eqn}{\end{eqnarray}}

\newcommand{\lin}{\llbracket}
\newcommand{\rin}{\rrbracket}

\newcommand{\ttsim}{\raise.17ex\hbox{$\scriptstyle\mathtt{\sim}$}}
\newcommand{\kh}{\kern .08em}

\newtheorem{pro}{Proposition} 

\newtheorem{lem}[pro]{Lemma}

\renewcommand{\thepro}{\arabic{pro}}

\newcommand{\comment}[1]{}

\newcommand{\lojasiewicz}{\L ojasiewicz }

\newcommand{\mut}{\nu_t}

\begin{document}

\title{Swarm gradient dynamics for global optimization: the density case}

\author{J\'{e}r\^{o}me Bolte\thanks{Toulouse School of Economics, University of Toulouse  Capitole, 1 esplanade de l'universit\'{e}, 31000 Toulouse, France, the authors acknowledge funding from  ANR under grant ANR-17-EUR-0010 (Investissements d'Avenir program)}, Laurent Miclo${}^*$\thanks{CNRS-IMT-TSE-R},   and St\'{e}phane Villeneuve${}^*$\thanks{TSE-TSM-R}}

\date{\today}

\maketitle

\begin{abstract}
Using jointly geometric and stochastic reformulations of nonconvex problems and exploiting a Monge-Kantorovich gradient system formulation with vanishing forces, we formally extend the simulated annealing method to a wide class of  global optimization methods. Due to an inbuilt  combination of a gradient-like strategy and particles interactions, we call them swarm gradient dynamics. As in the original paper of Holley-Kusuoka-Stroock, the key to the existence of a schedule ensuring convergence to a global minimizer is a functional inequality. One of our central theoretical contributions is the proof of such an inequality for one-dimensional compact manifolds. We conjecture the inequality to be true in a much wider setting. We also describe a general method allowing for global optimization and evidencing the crucial role of functional inequalities \`a la \L ojasiewicz.

\end{abstract}

\tableofcontents

\section{Introduction}
The global minimization of a non-convex function is one of the most challenging problems in modern optimization. There are few global optimization methods which provide reasonable convergence guarantees, the most famous are probably the simulated annealing, whose premises are found in \cite{metropolis53}, or the moment method \cite{lasserre01}, and their many variants. On the other hand, metaheuristics methods are numerous and have some notable empirical success: they orchestrate interactions between local and global strategies, combining random and deterministic procedures, and often ending up with methods using optimizing agents. Some examples of metaheuristics are inspired by analogies with biology, as evolutionary algorithms \cite{Fogel:2000},  ethology (e.g. ant colonies \cite{DorigoBlum:2005}), or particle swarms, see e.g., \cite{KennedyEberhart:1995}.  The goal of this paper is to introduce a new family of swarm methods through gradient descent in the Monge-Kantorovich space and  give general guarantees for their convergence to global minimizers.

Let us be more specific and consider the problem of solving
\begin{equation}\quad \min_M U,\tag{$\P$}\end{equation}
where  $U:M\to \R$ is a differentiable function defined on a compact Riemannian manifold $M$. In order to introduce our swarm methods, we need first some considerations on simulated annealing.
\paragraph{Three views on simulated annealing} Our starting point is indeed the famous si\-mu\-lated annealing method. In its time-continuous form, and when the state space is flat $M$ (e.g., a flat torus), it is a solution $X\df(X_t)_{t\geq 0}$ of the time-inhomogeneous Langevin-like stochastic differential equation 
\begin{equation}\label{langevin}
    dX_t=-\beta_t\nabla U(X_t)\,dt+\sqrt{2}dW_t,\tag{a}
\end{equation}
where $(W_t)_{t\geq 0}$ is a Brownian motion and $\beta_t\to +\infty$ \footnote{Often called the inverse temperature} is a time-dependent parameter tuned so that  the expectation of $U(X_t)$ tends to $\min_M U$. The formulation \eqref{langevin} can be extended to any compact Riemannian manifold $M$, but  this requires more involved notations, see e.g.\
Ikeda and Watanabe \cite{MR1011252} or Emery \cite{MR1030543}. 
\par
The intuitive interpretation is quite natural, the method combines local gradient search with a vanishing Brownian exploration of the feasible set $M$. 
Although the method is often used as a heuristics, its proof has been made rigorous in various  frameworks via different approaches, the two main ones being based on large deviations, see e.g.\  the compendium by Azencott et al.\ \cite{zbMATH00053884}, and on functional inequalities, cf.\ Holley, Kusuoka and Stroock \cite{MR995752}.
  Key to the foundational  approach of \cite{MR995752} , is the establishment of a generalized log-Sobolev inequality followed by hypercontractivity arguments.
This approach was then simplified by Miclo \cite{zbMATH04157608}, via the identification of the relative entropy as a convenient Lyapunov function.
In order to explain the role of the log-Sobolev inequality, the relative entropy, the mechanisms behind the convergence properties, and understand the scope of the method, we view simulated annealing along three complementary angles:
\begin{itemize}
\item[(a)] the SDE form of the algorithm: the overdamped Langevin dynamics (a) above,
\item[(b)] the PDE counterpart of (a) which describes the time evolution of the density $t\mapsto\rho(t)$ of $X_t$. It assumes the form of a Fokker-Planck equation,
\begin{equation}\label{FokkerPlanck}\tag{b}
\dt \rho =\beta_t \,\dive(\rho\nabla U)+\Delta \rho,\quad t\geq 0.
\end{equation}
 \item[(c)] Otto's formalism \cite{jko}
allows to interpret the latter as a gradient-like system in the  space of probabilities on $M$ endowed with Monge-Kantorovich metric:
 \begin{equation*}
 \tag{c}
\dt\,\rho(t)=-\gradw \, \U_{\beta_t}[\rho(t)],
\end{equation*}
\end{itemize}
with $\displaystyle \U_{\beta}[\rho]=\beta\int_M  U\rho\,d\ell+\int_{M} \rho\log \rho\,d\ell,$  
where $\ell$ is the Riemannian measure of $M$. This quantity is also known as the relative entropy of $\rho$ with respect to the Gibbs measure whose density is proportional to $\exp(-\beta U)$, which served as a Lyapunov function in \cite{zbMATH04157608}.

\smallskip
This triple perspective,  mainly due to \cite{jko}, is not new and has known a recent success in sampling \cite{dalalyan17,durmus17}, optimization \cite{Ma19,Ma21} and machine learning \cite{Raginsky17,bartlett}. 

\smallskip

Depending on the form we adopt to study the dynamics, subsequent results or developments may be considerably easier to understand. Indeed, while (a) classically provides operational algorithms through discretization, (b) offers a trac\-ta\-ble version  amenable to classical PDE analysis methods as Lyapunov methods. As for the last angle, (c), it confers a sharp geometrical content to the method and allows to interpret the essential tools of convergence through classical intuitive geometric ideas. An essential fact about (c) is that functional inequalities, as the log-Sobolev inequality, may be seen as \L ojasiewicz gradient inequalities. In our case it means that there exists an exponent  $\gamma\in (0,1)$ such that the slope of $(\U_\beta-\min\U_\beta)^\gamma$ is bounded away from zero (save at the stationary measure). This reparametrization sharpens the energy while leaving unchanged level sets: this allows for a direct convergence analysis of the gradient method (c), see \cite{blanchet18} and references therein for further insights.   In the simulated annealing case  the log-Sobolev inequality of Holley, Kusuoka and Stroock \cite{MR995752} turns out to be an instance of such an  inequality, see \cite{blanchet18}.

\medskip

\paragraph{Swarm gradient dynamics} The triple-perspective (a)-(b)-(c)  we used to describe the strategy of simulated annealing can be generalized to a much larger framework. For this, we adopt the angle (c) under which we observe that it is natural to consider more general convex functions $\varphi$ than $\RR_+\ni r\mapsto r\ln(r)$ in the Boltzmann entropy.  Referring to  the results in \cite{ags}, we may indeed use a whole family of convex functionals 
\begin{equation*}
\HH[\rho]=\int_M \varphi(\rho)\,d\ell,
\end{equation*}
leading to a penalized cost $$ \U_\beta[\rho]=\beta\int_M U\rho\,d\ell+\HH[\rho]$$
and to the triplet of ``equivalent''  minimizing dynamics modeled on (c), (b), (a),
\begin{align}\label{gradphi}
&\dt \rho (t) = -\gradw \,\U_{\beta_t}[\rho(t)]\\
\label{pdephi}
& \dt  \rho     =  \beta_t\dive(\rho\nabla U)+ \dive\rho \nabla \varphi'(\rho)\\
  &  dX_t  =-\beta_t\nabla U(X_t)\,dt+\sqrt{2}\alpha(\rho)dW_t,
       \end{align}
with $\lim_{t\to\infty}\beta_t=+\infty$, $\varphi,\alpha$ are some positive functions to be specified, and  $\rho$ is the law of $X$. The fact that a particle interacts with its law may be considered as a swarm effect, this is why we call these dynamics {\em swarm gradient dynamics}\footnote{Since $\beta$ is variable, they are actually time-dependent swarm dynamics.}, see Section \ref{s:swarm} for more insight.
Principles and other considerations behind the above dynamics are described in Sections \ref{s:globdyn} and \ref{s:swarm}.

\medskip

\paragraph{The key to convergence: functional inequalities} The central question is that of the  convergence properties to a global minimizer. In particular, an essential question is: what are assumptions ensuring that the global minimization $\min_M U$ problem is solved by the above? 

\smallskip

In simulated annealing,  the essential tool for convergence is the log-Sobolev inequality of Holley, Kusuoka and Stroock \cite{MR995752}. We also recalled that this inequality can be advantageously thought as a \L ojasiewicz inequality when considering the problem along the gradient system angle (c).  We follow therefore the same protocol but in a reverse way: we formally write \L ojasiewicz inequalities using the Monge-Kantorovich  formalism, which  reveals in turn the functional inequalities we would like to have at our disposal. This leads us to consider:
\begin{equation}\label{ineqf}
\int_M\vert \na \varphi'(\rho)-\na\varphi'(\mu_\beta)\vert^2\rho\, d\ell\geq  c(\beta)\,\Omega\lt(\int_M\varphi(\rho)-\varphi(\mu_\beta)-\varphi'(\mu)(\rho-\mu_\beta)\, d\ell\rt)
\end{equation}
where $c,\Omega:(0,+\infty)\to \R_+$ are positive functions having specific properties and where $\mu_\beta$ is the unique stationary measure of $\U_\beta$. 
  We are at the heart of this paper and our central result: proving such a functional inequality  under adequate assumptions.  Our result holds in compact one-dimensional manifolds for power-like potential function $\varphi$. As a consequence we obtain a full convergence result of our global methods on compact one-dimensional manifolds. We also evidence the general mechanisms of global convergence and for completeness we sketch the form that operational algorithms could take.
  
\smallskip

Apart from the interest of our work for optimization, we believe that it raises important questions and hopes on the validity domain of the family of inequalities in \eqref{ineqf}. Positive outcomes would lead to new results in optimization and in other fields.

\paragraph{Related works} The quantitative comparison between entropy-type functional and its time derivative, often called the entropy production or dissipation, dates back  to \cite{zbMATH04157608}, where it was exploited through the logarithmic Sobolev inequality of \cite{MR995752}. Using Otto's formalism this can be in turn reinterpreted as an approach \`a la \lojasiewicz \cite{blanchet18}.

Equation \eqref{pdephi} may be seen as a formal generalization of porous media equation and fast diffusion equations -- which correspond to the case when $\varphi$ is a potential. These have been studied by several authors using the Monge-Kantorovich framework, see, e.g.\, Otto \cite{otto01}, 
and {Carrillo}, {McCann} and  {Villani} \cite{zbMATH02109873,zbMATH05009584}.
 Their asymptotic analysis is through the Bakry-Emery  method \cite{MR889476}: it  consists in the second-order time differentiation of the entropy. Contrary to \cite{MR995752,zbMATH04157608} and our current approach, this  approach requires  convexity which makes it unsuitable for general global minimization.
 
 The  article of Iacobelli,  Patacchini and Santambrogio  \cite{iacobelli:hal-01860062} is also connected to our approach since they consider ultrafast diffusion equations which corresponds to a negative exponent $m$ in \eqref{varphim}. However, the hypocoercive bounds they obtain do not seem well-suited to extensions to time-inhomogeneous situations, since they do not lead to a differential inequality satisfied by the entropy-like functionals.

The uniqueness of the stationary measure, i.e. of the minimizer of $\U_\beta$,  is not  a new result, it can be found in, e.g.,\
{Carrillo}, {J\"ungel},  {Markowich},  {Toscani} and  {Unterreiter} \cite{zbMATH01661610}. For the sake of completeness,  we provide a proof in the next section. 

Let us conclude by mentioning a few works  using non-linear diffusion where non-linearities generally affects the drift coefficient but not the diffusion term as here. In Eberle, Guillin and Zimmer \cite{eberle:hal-01334806}, {Carrillo}, the authors use coupling techniques, {Gvalani}, {Pavliotis} and {Schlichting} \cite{zbMATH07169279} treat the case of interaction potentials while  Delarue and  Tse \cite{delarue2021uniform} consider chaos propagation.
 
 

\section{Presentation of the problem}\label{LeProbleme}
\subsection{A family of relaxations in the probability space $\PP(M)$}
Consider the non-convex minimization problem: 
\begin{equation}
\begin{array}{c}
\hbox{Find a global minimizer of $U:M\to\R$ on a compact Riemannian manifold $M$.}
\end{array}
\tag{$\PP$}\end{equation}
Denote by $d$ the distance on $M$, and $\ell$ the natural Riemannian measure. Up to a normalization factor,  assume $\ell(M)=1$.  Let $\P(M)$ be the space of probability measures on $M$ equipped with the Monge-Kantorovich  distance defined through
\begin{equation*}
\cW^2_2(\mu,\nu)=\inf \left\{\int_M d^2(x,y)\, p(dx,dy):\mbox{$p$ is a coupling of $\mu$ and $\nu$ on $M^2$}\right\},\label{monge}
\end{equation*} 
for any $\mu,\nu$ in $\P(M)$. The extreme values of $U$ play a special role in our approach,  one defines
\begin{equation}\label{osc}
\osc(U):=\max_M U-\min_M U,
\end{equation}
which we may assume positive --since otherwise the problem would be trivial.

\smallskip

We make the following regularity assumptions:
\smallskip

\noindent
{\bf Assumption (A).} The manifold $M$ and the function $U:M\to [0,+\infty[$ are of class $C^2$.\\

The $C^2$ regularity assumptions are simple means to obtain existence results for the gradient evolution in the Monge-Kantorovich space as in \cite{ags,ferreira18}

\smallskip
\noindent
We embed our problem in $\P(M)$ and consider  $\U[\rho]:=\int_MU\rho$ so that 
$$\min_{\P(M)} \U=\min U.$$
Let $\beta>0$, we introduce a penalized relaxation of $U$ in the metric space $(\P(M),\cW_2)$ through
\begin{equation}\label{penalty}
\U_\beta[\rho]=\beta \U[\rho] + \HH[\rho]=\beta \int_M U\rho\,d\ell + \HH[\rho]
\end{equation}
where
\begin{equation*}
\displaystyle\HH[\rho]=
\left\{\begin{aligned}
&\int_{M}\varphi(\rho) \, d\ell &&\mbox{ if $\rho$ is absolutely continuous w.r.t the Riemannian measure}\\
&  +\infty &&\mbox{ otherwise,}
\end{aligned}\right.
\end{equation*}
with $\varphi:[0,+\infty)\to \R_+$ is strictly convex and $C^2$ on $(0,+\infty)$. Up to a multiplicative factor, the first term in \eqref{penalty} is the classical relaxation of $U$ within the probability space over $M$. On the other hand, as in simulated annealing, the second term acts as a penalization forcing the minimizer to be unique and to have a density with respect to $\ell$ (see Lemma~\ref{varphiprime}).

\begin{remark} {\rm (a) (Power-like penalizations) A strong focus will be put on the class of power-like functions. For any $m\in(0,+\infty)\setminus\{1\}$, define the convex function $\varphi_m\st\RR_+\ri \RR_+$ via
\bqn{varphim}
 \fo r\geq 0,\qquad
 \varphi_m(r)&\df&
 \f{r^m-1-m(r-1)}{m(m-1)} \eqn
 Let us observe that $\varphi_m$ is a strictly convex function, $C^2$ on $(0,+\iy)$ and such that $\varphi_m(1)=0$, $\varphi'_m(1)=0$ and $\varphi''_m(1)=1$ for every admissible $m$. 
 By the Taylor-Lagrange formula, we deduce that $\varphi_m$ is always positive, except at 1.
 The convex function $\varphi_1\st\RR_+\ri \RR_+$ defined by $\varphi_1(r):=r\ln(r)-(r-1)$ is recovered as the limit of $\varphi_m$ when $m$ goes to $1$ and corresponds to the Boltzmann entropy. Hereafter, we will in particular consider functions $\varphi$ that are constructed by gluing together two different functions $\varphi_m$ at $1$.\\
 (b) (Regularity of the penalization) 
 Observe as well, from \cite[Theorem~9.3.9, p.212]{ags} and \cite[Proposition~9.3.2, p.210]{ags}, that the function $\HH$ is lower semi-continuous for the Monge-Kantorovich  distance,   and geodesically convex  in  $\PP(M)$. }
 \end{remark}
 
 \medskip
  
The penalization approach we adopt is through the one-parameter family of problems
\begin{equation}
\qquad \val (\PP_{\beta}):=\inf_{{\cal P}(M)}\U_\beta \tag{$\PP_{\beta}$}
\end{equation}
where the parameter $\beta>0$ is the inverse of a penalization parameter or the inverse of ``the temperature''  according to the simulated annealing literature.  It ultimately tends to $\infty$, and one has an elementary but important fact:

\begin{proposition}[A global optimization principle] Assume {\rm (A)} and that $\varphi:[0,+\infty)\to \R_+$ is strictly convex and $C^2$ on $(0,+\infty)$. Then
\begin{itemize} 
\item[(i)] $\displaystyle \lim_{\beta\to+\infty}\frac1\beta \,\val (\PP_{\beta})\, = \,\min_M U,$
\item[(ii)] if $\mu_\beta$ is a sequence of solutions to $(\P_\beta)$, the weak* limit points of $\mu_\beta$ have a support concentrated on the set of minimizers of $U$
\end{itemize}
\end{proposition}

\begin{proof}
As a first observation, it is clear that 
$$
\min_{\P(M)}\int_MU\rho=\min_M U.
$$
Fix  $\epsilon>0$ and choose $\rho_\epsilon$ to be an $\epsilon$-minimizer of $\U[\rho]=\int_MU\rho$. The above observation yields $\U[\rho_\epsilon]\leq\min _M U+\epsilon$. Since $\dom \HH$  contains smooth densities, one can also assume that  $\HH(\rho_\epsilon)$ is finite. Take $\beta>0$ and
let $\mu_{\beta,\epsilon}$ be an $\epsilon$-solution to $(\P_\beta)$, that is
$$\frac1\beta\U_\beta [\mu_{\beta,\epsilon}]=\U[\mu_{\beta,\epsilon}]+1/\beta \HH[\mu_{\beta,\epsilon}]\leq\U(\rho)+1/\beta \HH[\rho] +\epsilon $$ for all $\rho$. Thus choosing $\rho=\rho_\epsilon$ yields
 $$\frac1\beta\U_\beta[\mu_{\beta,\epsilon}]\leq \min_M U+1/\beta \HH[\rho_\epsilon]+2\epsilon.$$ 
 Letting $\beta$ goes to infinity yields $$\limsup_{\beta\to\infty}\frac1\beta\U_\beta[\mu_{\beta,\epsilon}]\leq \min_M U+2\epsilon.$$ Whence $\displaystyle \limsup_{\beta\to\infty,\epsilon\to0}\frac1\beta\U_\beta[\mu_{\beta,\epsilon}]\leq \min_M U$.  Since $\U_\beta\geq \min_M U$ by positivity of $\varphi$, (i) follows readily.
 
 \noindent
Let us prove (ii). Let $\rho$ be a limit point of $\mu_{\beta}$ for the weak* topology. Since $U$ is continuous, its mixed extension $\U$
is continuous for the weak* topology and thus $\lim_{\beta\to \infty} \U[\mu_\beta]= \U[\rho]$. On the other hand, by positivity of $\varphi$ one has 
$\U\leq \frac1\beta\U_\beta$, thus (i) gives $\limsup_{\beta\to \infty} \U[\mu_\beta]\leq \min_M \U$ whence $\U[\rho]\leq \min_M U$ and (ii) follows.\end{proof}

\medskip

\noindent
Observe that
\begin{equation}
\app (\beta)=\frac1\beta\inf_{\P(M)} \U_\beta-\min_M U\label{app}
\end{equation}
tends to zero when $\beta$ tends to $+\infty$ by the previous result. The quantity $\app (\beta)$ measures the approximation abilities of the problem $(\P_\beta)$ with respect to the initial problem $\min_M U$.

 \subsection{Variational considerations and stationary measures}\label{s2.2}
 Let us analyze the first-order conditions for the above problem $(\PP_{\beta})$ through the lenses of the Monge-Kantorovich metric. We shall use freely the definition of Monge-Kantorovich  subgradients and related objects. As they are only central to our understanding but not to our proofs, we refer to \cite{ags} for details. The subgradient of $\U_\beta$ with respect to the Monge-Kantorovich  metric has a domain contained in $L^1(M)$ and is formally given by
 \bq
\gradw\; \U_\beta\,[\rho]&=&-\dive(\rho(\beta \na U+\na \varphi'(\rho)))\eq
for any admissible $\rho$ in $L^1(M)$.\\
Stationary solutions of \eqref{pdephi}, with $\beta_t\equiv\beta$, are thus probability densities $\mu$ solution to
\bqn{dive}
\dive(\mu(\beta\na U+\na \varphi'(\mu)))&=&0,
\eqn
which is to be understood  in the standard weak sense. By integration by parts, we have for $f \in C^2(M)$,
\bq
\int_M \dive(\mu\na \varphi'(\mu)) f\, d\ell
&=&-\int_M \lan \mu \na\varphi'(\mu),\na f\ran\, d\ell\\
&=&-\int_M \mu\varphi''(\mu)\lan \na \mu,\na f\ran\, d\ell\\
&=&-\int_M \lan \na(\mu\alpha(\mu)),\na f\ran\, d\ell\\
&=&\int_M \mu\alpha(\mu) \tr f\, d\ell\\
&=&\int_{\{\mu>0\}} \alpha(\mu)\tr f\, \mu d\ell
\eq
where the function $\alpha\st(0,+\iy)\ri\RR_+$ is given by
\bq
\fo r> 0,\qquad \alpha(r)&\df &\f1r\int_0^r s\varphi''(s)\, ds.
\eq
On the other hand, we have
\bq
\int_M \dive(\mu \beta\na U) f\, d\ell=-\int_M \lan \beta \na U, \na f\ran \mu\,d\ell.
\eq
Finally, $\mu$ is a stationary solution to \eqref{dive} if
\bq
\fo f\in\cC^2(M),\qquad \int_{\{\mu>0\}} L_\mu[f]\, \mu d\ell&=&0\eq
with
\bq
L_\mu[f]&=&\alpha(\mu)\tr f-\lan \beta \na U,\na f\ran.\eq
\begin{remark}[Infinitesimal generator]{\rm 
Observe that the choice $\varphi=\varphi_1$ leads to $\alpha(r)=1$, that is to
$$
L_\mu[f]=\tr f-\lan \beta \na U,\na f\ran.
$$
This is the infinitesimal generator of the classical (overdamped) Langevin SDE. Up to a multiplicative constant, this is the only choice for the operator $L_\mu$ to be independent of $\mu$.}
\end{remark} 

Let us apply formally the above relationship to the function $f=\beta U +\varphi'(\mu)$, assuming here that the stationary density $\mu$ is smooth enough. We obtain\footnote{One may observe that the first term in \eqref{grad} is the squared norm of the Monge-Kantorovich  gradient of $\U_\beta$ evaluated at $\mu$.}
\begin{align}\label{grad}
\int_{\{\mu>0\}} |\na (\beta U +\varphi'(\mu))|^2\mu\,d\ell &=-\int_{\{\mu>0\}} L_\mu[\beta U +\varphi'(\mu)]\, \mu d\ell\\\notag
&=0,
\end{align}
where the last equality comes from the stationarity of $\mu$. Therefore, $\beta U + \varphi'(\mu)$ is constant on every connected component of the set $\{\mu>0\}$. We analyze this condition in the next paragraph.

\subsection{Uniqueness of the stationary density}
The main ingredient to study the uniqueness of the stationary density is the relation
\bqn{cO}
\beta U+\varphi'(\mu)&=&c,
\eqn
where the constant $c$ depends on the considered  connected component of the support~$\{\mu>0\}$.

\smallskip

Define $I\df\varphi'((0,+\infty))$ and denote by $\psi\st I\ri (0,+\iy)$ the inverse of $\varphi'$.

\smallskip

\begin{lem}[Existence and uniqueness of the minimizer of $(\P_\beta)$]\label{varphiprime}
Assume $\varphi'(0)=-\infty$, then there exists an unique stationary density $\mu_\beta$ solution to \eqref{dive}. Moreover, 
\begin{itemize}
    \item[(i)]
$\mu_\beta$ is positive everywhere on $M$ and is characterized by the relation
\bqn{rhopsicU}
\mu_\beta&=&\psi(c^*-\beta U),\eqn
where $c^*$ is a normalization parameter characterized by the condition\\
$$\int_M \psi(c^*-\beta U)\, d\ell=1.$$ 
\item[(ii)] $\mu_\beta$ is the global minimizer of $\U_\beta$.
\end{itemize}
\end{lem}

\begin{proof} 
 Let us start with (i) and by showing that a stationary density $\mu$ is everywhere positive. Towards a contradiction, assume that the set $\{\mu=0\}$ is non-empty and let $M_1$ be a connected component of the open set $\{\mu>0\}$ with $\pa M_1\neq \emptyset$. Let $(x_n)_{n\in\NN}$ be a sequence of elements of $M_1$ converging to a point of the boundary $\pa M_1$. According to \eqref{cO}, we have for every $n\in\NN$
 \bq
 \beta U(x_n)+\varphi'(\mu(x_n))&=&c_{M_1}\eq
where the left-hand side term converges to $-\infty$ which  is absurd.
Therefore, $M_1=M$ and equation \eqref{cO} is valid everywhere on $M$. Set $c\df c_{M_1}$.
Being strictly convex, the function $\varphi'$ is one-to-one and onto between $(0,+\iy)$, and by definition of $\psi$,
equation \eqref{cO} rewrites:
\bqn{rhopaO}
 \mu&=& \psi(c^*-\beta U)\eqn
\par
\noindent
As $\mu$ is a density function, we must have
\bqn{norma}
\int_M \psi(c^*-\beta U)\, d\ell&=& 1.\eqn
Since $\psi$ is strictly increasing and satisfies $\lim_{\inf I} \psi=\lim_{-\infty} \psi=0$ and $\lim_{\sup I} \psi=+\infty$, there is a unique value  $c^*\in\RR$ that satisfies equation \eqref{norma}.
We have 
$$
\mu=\psi(c^*-\beta U)$$
which ends not only the proof of the uniqueness of the stationary density $\mu$ but also gives its existence and its explicit form. To see (ii) and that $\mu$ is the global minimizer, we observe that
\bq
\U_{\beta}[\rho]-\U_\beta[\mu]&=&\int_M(\varphi(\rho)-\varphi(\mu))\, d\ell+\int_M \beta U\,(\rho-\mu) \,d\ell\\
&=& \int_M(\varphi(\rho)-\varphi(\mu))\, d\ell+\int_M (c^*-\varphi'(\mu))\,(\rho-\mu) \,d\ell\\
&=& \int_M\left( \varphi(\rho)-\varphi(\mu)-\varphi'(\mu)\right)\,(\rho-\mu) \,d\ell,
\eq
which is positive whenever $\rho\neq\mu$ by strict convexity of $\varphi$.
\end{proof} 

\smallskip

When no confusion can occur we simply write $\mu$ for $\mu_\beta$. We gather the assumptions we need regarding $\varphi$ within
\medskip

\noindent
{\bf Assumption (B).} $\varphi:[0,+\infty)\to\R$ is convex, twice differentiable on $(0,+\infty)$ with $\varphi''>0$, and satisfies 
$\varphi'(0)=-\infty$. \\
\smallskip

\subsection{Global minimization dynamics} \label{s:globdyn}
\paragraph{Minimizing dynamics} We are now in position to provide dynamical systems meant to solve the problem $(\P)$. Inspired by Holley, Kusuoka and Stroock's approach to simulated annealing \cite{MR995752}, as it was simplified in \cite{zbMATH04157608},  and using as well the gradient view provided by Otto's formalism, we consider, formally, the gradient system 
\begin{equation}\label{grad-sys}
\quad\dt \rho(t)=-\gradw \,\U_{\beta_t}\,[\rho(t)] \mbox{ a.e. on $\R_+$,} \quad \rho(0)=\rho_0,
\end{equation} 
where the term $\beta_t$ is a $C^1$ positive time-varying parameter and where we use Newton's notation, $\dt \rho$ here, for  time derivatives. The initial distribution $\rho_0$ is chosen in the domain of $\U_\beta$, that is in the domain of $\HH$. 
The time dependent density $\rho$ turns out to satisfy the following partial differential equation
\bqn{evol0}
\qquad\dt  \rho&=&\dive(\rho \,(\beta_t\na U))+\dive(\rho(\na \varphi'(\rho))), \quad \rho(0)=\rho_0.\eqn
The time-varying parameter $\beta_t$ is traditionally interpreted as an inverse of a temperature which typically cools down, i.e., \begin{equation}\label{infty}\lim_{t\to \iy}\beta_t=+\infty\end{equation} Here we also interpret this parameter as the inverse of a penalty term echoing the static formula \eqref{penalty}.
\par
\begin{remark}{\rm (a) (Simulated annealing) When $\varphi=\varphi_1$, \eqref{evol0} boils down to the famous simulated annealing dynamics 
\bqn{siman}
\dt  \rho &=&\beta_t\dive(\rho\na U)+\Delta\rho\eqn
which, by a famous ``nonconvex''  extension of the log-Sobolev inequality, due to Holley, Kusuoka and Stroock \cite{MR995752}, is known to generate measures concentrating on the set of global minimizers of $U$ whenever the temperature schedule is finely tuned.\\
(b) (Porous media) Taking $U$ constant and $\varphi=\varphi_m$ in \eqref{evol0} with $m >0$, the dynamic cor\-res\-ponds to the porous media equation
$$
\dt \rho=m\Delta \left( \rho^m \right).
$$
The case $m>1$ refers to the slow diffusion case while the case $m<1$, for which $\varphi'(0)=-\infty$ refers to the fast diffusion situation, (see Vazquez \cite{vazquez17} and Otto \cite{otto01})}.

\end{remark}

\paragraph{Existence results and evolution equations} Following the pioneering work of  \cite{ags}, 
the non-autonomous theory for Monge-Kantorovich  gradient flows has recently been developed in \cite{ferreira18}. In the line  of \cite[Theorem 4.4, Theorem 5.4]{ferreira18} and the existence results of \cite{iacobelli:hal-01860062}, we {\em assume} that \eqref{grad-sys} and \eqref{evol0}, have a common unique solution curve $t\mapsto\rho(t)$ in $(\P(M),\cW_2)$, which satisfies in addition
\begin{align}
& t\mapsto \U_\betat[\rho(t)] \mbox{ and $t\mapsto\rho(t)$ are absolutely continuous,}\label{ac0}\\
& \dt\,\U_\betat[\rho(t)]=\int_M(\varphi'(\rho)+ \beta_t U)\dt {\rho}\, d\ell+\dt \betat\int_MU\rho\,d\ell,\label{ac1}
\end{align}
where the time derivatives are taken for almost all times.

\paragraph{Functional inequalities} 

Under hypothesis {\bf (A), (B)} and some extra-assumptions on $\varphi$ related to the geometry of the penalized cost, we intend to prove that the dynamics  \eqref{grad-sys}-\eqref{evol0} has global optimizing properties,  in the sense that the global cost
\begin{align}
 \U_{\beta_t}[\rho]=\int_M\left(\varphi(\rho)+ \beta_t U\rho\right) \,d\ell
 \end{align}
evaluated along the trajectory $t\mapsto\rho(t)$ given by \eqref{grad-sys}-\eqref{evol0} should converge to the value of $(\PP)$, i.e.,
$$\lim_{t\to+\infty}\U_{\beta_t}[\rho]=\mbox{val} ( \PP)=\min_M U.$$

As it is customary in the analysis of PDEs the key to convergence is given by ``entropy-energy'' or ``entropy-production''  functional inequalities. In the ``gradient or in the optimization world'' , these can often be seen as \L ojasiewicz type inequalities, see  \cite{blanchet18} and references therein. They connect the cost $\U_\beta$ to the norm of its gradient $\|\gradw\U_\beta\|$ and to the constant $\beta$: 
\begin{equation}
\label{lojalike}\|\gradw\U_\beta\|^2\geq c(\beta)\,\Omega\Big( \U_\beta(\rho)-\min \U_\beta\Big),
\end{equation}
where $c,\Omega:(0,+\infty)\to\R_+$ are positive functions, with $\Omega$ being increasing and null at zero.
\noindent
Reexpressing $\U_\beta$ by means of its stationary density \eqref{rhopsicU} gives
\bq
\U_{\beta}(\rho)-\U_{\beta}(\mu)&=&\int_M\varphi(\rho)-\varphi(\mu)\, d\ell+\int_M \beta U\,(\rho-\mu) \,d\ell\\
&=& \int_M\varphi(\rho)-\varphi(\mu)\, d\ell+\int_M (c^*-\varphi'(\mu))\,(\rho-\mu) \,d\ell\\
&=& \int_M\left[\varphi(\rho)-\varphi(\mu)-\varphi'(\mu)\,(\rho-\mu)\right] \,d\ell.
\eq
Because $\varphi$ is convex, we obtain $\U_{\beta}(\rho) \ge \U_{\beta}(\mu)$ thus $\U_{\beta}(\mu)=\min\U_\beta$.\\
As a consequence, inequality \eqref{lojalike} writes
\begin{equation}\label{KL}
\int_M\vert \na \varphi'(\rho)-\na\varphi'(\mu)\vert^2\rho\, d\ell\geq  c(\beta)\, \Omega\lt(\int_M\varphi(\rho)-\varphi(\mu)-\varphi'(\mu)(\rho-\mu)\, d\ell\rt)
\end{equation}
where $c,\Omega:(0,+\infty)\to \R_+$ are positive functions.
A typical example is given by the log-Sobolev inequality of Holley, Kusuoka and Stroock which can be written as 
\begin{equation}\label{logsob}
\int_M\vert \na \varphi'(\rho)-\na\varphi'(\mu)\vert^2\rho\, d\ell\geq  C_{\rm \tiny{HKS}}(\beta)\lt(\int_M\varphi(\rho)-\varphi(\mu)-\varphi'(\mu)(\rho-\mu)\, d\ell\rt)
\end{equation}
where $\varphi(r)=\varphi_1(r)=r\ln(r)-r+1$ and
\bq
\lim_{\beta\ri+\iy}\f1\beta\ln(C_{\rm \tiny{HKS}}(\beta))&\geq &-\osc(U)\eq
(see \cite{MR995752} for the precise description of the l.h.s.\ in terms of the landscape of $U$).
\par
When $U$ is convex and $\varphi$ is power-like, one can also recover Gagliardo-Nirenberg inequalities of \cite{delpino02}, see \cite{blanchet18} for connections with \L ojasiewicz inequalities.

\medskip

\paragraph{Convergence mechanisms for a fixed penalization parameter} As previously mentioned, we adapt the approach of \cite{zbMATH04157608} developed in the Boltzmann entropy case ($\varphi=\varphi_1$) to our generalized class of relaxations.
 
 \smallskip
 
  Let us provide a first account of the general method through the constant parameter case. 
For $\rho\in \P(M)$ having a density with respect to $\ell$, set 
\begin{align}
& \I[\rho]=\U_\beta(\rho)-\min_{\P(M)}\U_\beta\\
& \J[\rho]=\int_\TT\vert \na \varphi'(\rho)-\na\varphi'(\mu)\vert^2\rho\, d\ell
\end{align}
where the quantities may take infinite values and where $\mu$ is the unique stationary density (see  Lemma \ref{varphiprime}), so that $\I[\mu]=0$.

\medskip

 At this stage, we do not assume that $\beta>0$ depends on time.
 
 \medskip
By time differentiation, using \eqref{ac0}-\eqref{ac1} and the evolution equation, we obtain
\bq
\dt\, \I[\rho(t)]&=&\int_M\varphi'(\rho)\dt {\rho}\, d\ell+\int \beta U\dt {\rho}\, d\ell\\
&=&\int_M(\varphi'(\rho)+\beta U)\dt {\rho}\, d\ell\\
&=&\int_M(\varphi'(\rho)+\beta U)\dive(\rho(\beta \na U+\na \varphi'(\rho)))\, d\ell\\
&=&-\int_M\na (\varphi'(\rho)+\beta U)(\beta \na U+\na \varphi'(\rho)))\, \rho d\ell\\
&=&-\int_M\vert \na \varphi'(\rho)+\beta\na U\vert^2\, d\rho.
\eq
Whence, if we have some inequality \`a la \L ojasiewicz like \eqref{KL} (as for instance the log-Sobolev inequality of Holley, Kusuoka and Stroock \cite{MR995752} when $\varphi=\varphi_1$), we derive a differential inequality for the one-variable function $\I[\rho]$,
\begin{equation} \label{liap}
\dt\, \I[\rho(t)] \: \leq  \: -c(\beta)\,\Omega(\I[\rho(t)])
\end{equation}
This implies in turn that $\I[\rho(t)]$ converges. If this limit was not zero, the fact that $\Omega$ is positive out of $0$ would imply, through \eqref{liap}, that the decrease-rate would be perpetually lower than a negative constant, which is absurd. This allows to prove that $\I[\rho(t)]$ tends to zero as $t\to\infty$. We thus have proved the first part of:

\begin{theorem}[Convergence with a non-vanishing penalty parameter] Assume  that $\bf (A), (B)$ are satisfied, and that there exist $c>0$, $\Omega:\R_+\to\R_+$ increasing, such that  an inequality of the type
\begin{eqnarray}\label{fineq}
\int_M\vert \na \varphi'(\rho)-\na\varphi'(\mu)\vert^2\rho\, d\ell & \geq &  c\,\Omega\lt(\int_M\varphi(\rho)-\varphi(\mu)-\varphi'(\mu)(\rho-\mu)\, d\ell\rt),
\end{eqnarray}
holds true whenever $\rho$ is measurable and the left hand side is finite. Then 
\begin{enumerate}[(i)]
\item $\U_\beta[\rho(t)]\to \min \U_\beta$. 
\item  If moreover $\Omega(s)=\Theta(s^{2\theta})$ at $0$, with $\theta\in (0,1)$, then $\rho(t)$ tends to $\mu_\beta$ for the Monge-Kantorovich  metric, i.e., for the weak* topology.
\end{enumerate}
\end{theorem}
\begin{proof} Item (i) is already proved. For (ii), from $\Omega(s)=\Theta(s^{2\theta})$ holds, we deduce that the  lower semicontinuous function $\U_\beta$ satisfies a \L ojasiewicz inequality as in \cite[Theorem 2]{blanchet18}, so that one may assert  that the curve $t\mapsto\rho(t)$ has a finite Monge-Kantorovich  length; convergence rates depending on $\theta$ are also available. \end{proof}
\medskip

What are the conditions for  the above inequality \eqref{fineq} to be valid, is a delicate open question. The subject of the next section, and the central result of this paper, is to establish  such inequalities for one dimensional compact manifolds and a family of power-like potentials~$\varphi$.

\section{A functional inequality on the circle}

\subsection{Main theorem} 

For $m\in (0,1/2)$, set
  \bqn{varphi}
 \fo r\geq 0,\qquad
 \varphi_{m,2}(r)&\df&\lt\{\begin{array}{ll}
\varphi_{m}(r)&\hbox{ if $r\in(0,1]$,}\\[3mm]
\varphi_{2}(r)=\frac{(r-1)^2}{2}&\hbox{ if $r\in(1,+\iy)$.} 
 \end{array}
 \rt.\eqn
 The function $\varphim$ 
 is convex on $[0,+\iy)$ and $\cC^2$ on $(0,+\iy)$. The latter property is a consequence of the fact that
  \bq
 \varphi_m(1)&=&0\\
 \varphi'_m(1)&=&0\\
 \varphi''_m(1)&=&1.\eq
Observe also that $\varphim'$ is concave on $(0,+\infty)$ with $\varphim'(0)=-\infty$, so that Lemma \ref{varphiprime} applies.  Let us recall that the unique solution of $(\P_\beta)$ is denoted by $\mu=\mu_\beta$ and that $\psim=[\varphim']^{-1}$. 

\medskip

This section is devoted to the proof of:
\begin{theorem}[A new functional inequality on the circle]\label{funcineq}
Assume that $M$ is the circle $\TT\df\RR/(L\ZZ)$ of perimeter $L>0$ endowed with its usual Riemannian structure.
Then there exists a constant $c(\beta)$, depending on $\mu_{\min}$ and $L$, 
  such that
for any measurable density $\rho$ on $\TT$
 \begin{eqnarray*}
\int_\TT\vert \na \varphim'(\rho)-\na\varphim'(\mu)\vert^2\rho\, d\ell \geq  c(\beta)\,\Omega\lt(\int_\TT\varphim(\rho)-\varphim(\mu)-\varphim'(\mu)(\rho-\mu)\, d\ell\rt)
 \end{eqnarray*}
where
\begin{eqnarray*}
c(\beta) & = & \displaystyle O\left(\beta^{\frac{-3(2-m)}{1-2m}}\right)\\
& &\notag \\
\Omega(r)&=& 
\left\{\begin{aligned}
&r^{\f32}&&  \hbox{ if $r\in[0,1)$}\\
\\
 & r^{\f{1-2m}{2(1-m)}} && \hbox{ if $r\geq 1$}.
 \end{aligned}\right.
\end{eqnarray*}

\end{theorem}

\begin{corollary}[An inequality \`a la Talagrand]\label{c:tal}
Under the assumptions of the previous theorem,
for any measurable density $\rho$ on $\TT$, 
 \begin{eqnarray*}
\int_\TT\varphim(\rho)-\varphim(\mu)-\varphim'(\mu)(\rho-\mu)\, d\ell \geq  d(\beta)\,\omega\Big(\M_2(\rho,\mu)\Big)
 \end{eqnarray*}
where $d(\beta)= \displaystyle O\left(\beta^{\frac{-3(2-m)}{1-2m}}\right)$, and
\begin{eqnarray}
\omega(r)&=& 
\left\{\begin{aligned}
&8/5 \; r^{\f58}&&  \hbox{ if $r\in[0,1)$}\\
\\
 & 4(1-m)/(3-2m)\; r^{\f{3-2m}{4(1-m)}} && \hbox{ if $r\geq 1$}.
 \end{aligned}\right.
\end{eqnarray}
\end{corollary}

\bigskip

The rest of this section is devoted to the proof of this theorem (the corollary will follow easily using a generalization of Otto-Villani theorem \cite{blanchet18}). Most of the intermediary results we provide are valid for a general compact $C^2$ manifold, thus, unless otherwise stated, we assume for the moment that $M$ is arbitrary. 
In the remaining subsections of the present section, for the sake of simplicity, we shall often write $\varphi=\varphim$ and $\psim=\psi$.
\subsection{Some estimates}

Let us define the positive  quantities
\bqn{boundsformu}
\mu_{\min}\df \min_M\mu\,\hbox{ and }\,
\mu_{\max}\df\max_M\mu.
\eqn
\begin{pro}[Bounds for the stationary measure]\label{controlwedge}
We have, for any $0<m<1$,
\bq
(1+(1-m)\beta \osc(U))^{\frac{1}{m-1}} \leq \ \mu_{\min}\ \leq \ \mu_{\max}\ \leq \beta\osc(U)+1.
\eq
where we recall that $\osc(U)=\max_MU-\min_MU$.
\end{pro}
\begin{proof} 
The stationary measure $\mu$ satisfies for every $x \in M$,
 $\mu(x)=\psi(c^*-\beta U(x))$ for some real constant $c^*$ and with $\beta >0$ (recall \eqref{rhopaO}). Because $\psi$ is nondecreasing , we have
$$
\forall x \in M,\,  \psi(c^*-\beta \max_MU)\le \mu(x) \le \psi(c^*-\beta \min_M U).
$$
Integrating with respect to the probability measure $\ell$, we obtain
$$
 \psi(c^*-\beta \max_MU)\le 1 \le \psi(c^*-\beta \min_M U).
$$
Because $\varphi^\prime$ is nondecreasing and $\varphi^\prime(1)=0$, we obtain the following bounds for the constant~$c^*$
$$
\beta \min_M U \le c^* \le \beta \max_M U.
$$
Finally, for every $x \in M$, $\psi(-\beta\osc(U))\le \mu(x) \le \psi(\beta \osc(U))$. Because $\beta \osc(U) \ge 0$, we have
for any $m\in(0,1)$,
$$
(1-(m-1)\beta \osc(U))^{\frac{1}{m-1}} \leq \ \mu_{\min}\ \leq \ \mu_{\max}\ \leq 1+\beta\osc(U),
$$
which ends the proof.\end{proof}

\medskip

The following formal observation is essential. When the potential function $\varphi$ is the entropy function $\varphi_1$, the density of $\rho$ with respect to $\mu$ plays a pivotal role in the establishment of the log-Sobolev inequality \eqref{logsob}, see \cite{MR995752}. In our case,
the counterpart is the function
\bq
\psi(\varphi'(\rho)-\varphi'(\mu)).\eq
It is also convenient to use the quantity
\bqn{g}
g:=\ \varphi'(\rho)-\varphi'(\mu)\eqn
so that
\bqn{rho}
\rho&=&\psi(g+\varphi'(\mu)).\eqn
\par

\smallskip

\paragraph{An upper bound for the reduced cost $\I[\rho]$}
This necessitates three steps.
\begin{lem}\label{step1}
For any measurable density $\rho$, we have
\bq
\varphi(\rho)-\varphi(\mu)-\varphi'(\mu)(\rho-\mu)&\leq &
\varphi(\psi(g+\varphi'(\mu_{\max})))-\varphi(\mu_{\max})-\varphi'(\mu_{\max})(\psi(g+\varphi'(\mu_{\max}))-\mu_{\max})\eq
\end{lem}
\begin{proof} 
By definition of $g$, we have
\bq
\varphi(\rho)-\varphi(\mu)-\varphi'(\mu)(\rho-\mu)&= &
\varphi(\psi(g+\varphi'(\mu)))-\varphi(\mu)-\varphi'(\mu)(\psi(g+\varphi'(\mu))-\mu).\eq
Fix $x\in M$ and consider the function $F$ defined on $(0,+\iy)$ by
\bq
\fo r>0,\qquad 
F(r)&\df& \varphi(\psi(g(x)+\varphi'(r)))-\varphi(r)-\varphi'(r)(\psi(g(x)+\varphi'(r))-r)\eq
To prove the result, it is sufficient to show that $F$ is nondecreasing . 
For $r>0$, we compute 
\bq
F'(r)&=&\varphi'(\psi(g(x)+\varphi'(r)))\psi'(g(x)+\varphi'(r)))\varphi''(r)-\varphi'(r)-\varphi''(r)(\psi(g(x)+\varphi'(r))-r)
\\&&
-\varphi'(r)[\psi'(g(x)+\varphi'(r)))\varphi''(r)-1]\\
&=&\lt[\varphi'(\psi(g(x)+\varphi'(r)))\psi'(g(x)+\varphi'(r))-(\psi(g(x)+\varphi'(r))-r)-\varphi'(r)\psi'(g(x)+\varphi'(r))\rt]\varphi''(r)\\
&=&\lt[(g(x)+\varphi'(r))\psi'(g(x)+\varphi'(r))-(\psi(g(x)+\varphi'(r))-r)-\varphi'(r)\psi'(g(x)+\varphi'(r))\rt]\varphi''(r)\\
&=&\lt[g(x)\psi'(g(x)+\varphi'(r))-(\psi(g(x)+\varphi'(r))-r)\rt]\varphi''(r)\\
&=&\lt[g(x)\psi'(\varphi'(s))-(s-r)\rt]\varphi''(r)
\eq
where we set
$s\df \psi(g(x)+\varphi'(r))$
in the last equality. Because $\psi=(\varphi')^{-1}$, we have 
\bq
\psi'(\varphi'(s))&=&\f1{\varphi''(s)}\eq
and we get for $r>0$,
\bq
F'(r)&=&[g(x)-\varphi''(s)(s-r)]\f{\varphi''(r)}{\varphi''(s)}\\
&=&[\varphi'(s)-\varphi'(r)-\varphi''(s)(s-r)]\f{\varphi''(r)}{\varphi''(s)}.
\eq
Because, the function $\varphi$ is convex and $\varphi'$ is concave, we have ${\varphi''(r)}/{\varphi''(s)}$ is positive and the quantity 
$\varphi'(s)-\varphi'(r)-\varphi''(s)(s-r)$ is non-negative, so that $F'\geq 0$ on $(0,+\iy)$
and $F$ is thus nondecreasing .\end{proof} 

\medskip

We define further
$
\xi_{\max}(s)\df\psi(s+\varphi'(\mu_{\max}))
$ for any real number $s$
and set $\rho_{\max}\df\xi_{\max}(g)$. Therefore, Lemma \ref{step1} can be rewritten
\bqn{plusjoli}
\varphi(\rho)-\varphi(\mu)-\varphi'(\mu)(\rho-\mu)\le
\varphi(\rho_{\max})-\varphi(\mu_{\max})-\varphi'(\mu_{\max})(\rho_{\max}-\mu_{\max}).
\eqn
We are in  position to give an upper bound for $\I[\rho]$.
\smallskip
\begin{lem}[An upper bound for the reduced cost $\I$]\label{step2}
For every $\rho \in \P(M)$ such that $g\in L^2(\ell)$, we have
$$
\I[\rho]\le \int_M g^2(x)\,\ell(dx).
$$
\end{lem}
\begin{proof} 
Because $\varphi$ is convex, we have
\bq
\varphi(\mu)-\varphi(\rho)-\varphi'(\rho)(\mu-\rho)&\ge& 0 
\eq
and
\bq
\varphi(\mu_{\max})-\varphi(\rho_{\max})-\varphi'(\rho_{\max})(\mu_{\max}-\rho_{\max})&\ge& 0.
\eq
Adding the latter positive quantity to the right-hand-side of equation \eqref{plusjoli} gives
\bq
0 \le \varphi(\rho)-\varphi(\mu)-\varphi'(\mu)(\rho-\mu)&\le&
(\varphi'(\rho_{\max})-\varphi'(\mu_{\max}))(\rho_{\max}-\mu_{\max})\\
&=& g(\xi_{\max}(g)-\xi_{\max}(0)).
\eq
Recalling that $\varphi$ has been constructed by gluing $\varphi_m$ and $\varphi_2$ at 1 (see equation \eqref{varphi}), we have that $\psi=(\varphi')^{-1}$ is convex and increasing with $0\le \psi' \le 1$. Therefore, $\xi_{\max}$ is convex and we have
$$
\psi'(\varphi'(\mu_{\max}))g \le \xi_{\max}(g)-\xi_{\max}(0) \le \psi'(g+\varphi'(\mu_{\max}))g.
$$
Whence,
$$
|\xi_{\max}(g)-\xi_{\max}(0)|\le |g|,
$$
which gives the desired result.\end{proof}   

\paragraph{A lower bound for the squared Monge-Kantorovich  gradient $\J[\rho]$} Once more several steps are necessary to obtain a bound. 
Let us define the function $\theta:\R\to\R$ via
\bqn{theta}
\fo r\in\RR,\qquad \theta(r)&=&\int_{0}^r \sqrt{\psi(s+ \varphi'(\mu_{\min})) }\, ds.\eqn

We observe first that:
\begin{lem}[Lower bound for $\theta$]\label{secondstep}
Assume $0 < m < \frac{1}{2}$. For any $r\in \R$,
$$
|\theta(r)|\ge  \frac{2}{3} \left(\min\left(\frac{1}{(C_0+(1-m))^{\frac{3}{2}-\eta}},\sqrt{\psi^\prime(\varphi'(\mu_{\min}))}\right)\right)\min(|r|^{3/2},|r|^\eta),
$$ where $C_0=1-(1-m)\varphi^\prime (\mu_{\min}) $ and $\eta:=\frac{1-2m}{2(1-m)} \in (0,1/2)$.

\end{lem}
\begin{proof}  

Assume first that $r>0$. Because $\psi$ is convex, we have for every $s$, $$\psi( s+\varphi'(\mu_{\min})) \ge \mu_{\min} +\psi^\prime(\varphi'(\mu_{\min}))s.$$ But $ \mu_{\min}>0$ thus
\begin{align*}
\theta(r)&\ge\sqrt{\psi^\prime(\varphi'(\mu_{\min}))}\int_0^r \sqrt{s}\,ds\\
&=\frac{2}{3}\sqrt{\psi^\prime(\varphi'(\mu_{\min}))} r^{3/2}.
\end{align*}
Now, assume that $r<0$. By a change of variables, $t=-s$ and $\tau=-r>0$, we get
$$
-\theta(r)=\int_{0}^\tau \sqrt{\psi( \varphi'(\mu_{\min})-t) }\, dt.
$$
Remembering that $\psi$ is both convex and positive, we get, for all $t$,
$$
\psi( \varphi'(\mu_{\min})-t)\ge \psi^\prime( \varphi'(\mu_{\min})-\tau)(\tau-t).
$$
We deduce that
$$
-\theta(r) \ge \frac{2}{3} \sqrt{\psi^\prime( \varphi'(\mu_{\min})-\tau) } r^{3/2}.
$$
We shall now use the explicit form of the derivative of $\psi$ given by

\bqn{deriveepsi}
 \psi'(\tau)&\df&\lt\{\begin{array}{ll}
(1+(m-1)\tau)^{\frac{2-m}{m-1}}&\hbox{if $\tau\in(-\iy,0)$}\\[3mm]
1&\hbox{if $\tau\in(0,+\iy)$} 
 \end{array}
 \rt.
\eqn
By definition of $\mu_{\min}$, we have $\mu_{\min} <1$ whenever $U$ is not constant over $M$. Therefore  $\varphi'(\mu_{\min}) \le 0 <\tau$, so that
$$
\sqrt{\psi^\prime( \varphi'(\mu_{\min})-\tau) }=\left( 1-(1-m) \varphi'(\mu_{\min})+(1-m)\tau \right)^{\frac{2-m}{2(m-1)}}.
$$
Set
$$
C_0=1-(1-m)\varphi^\prime (\mu_{\min}),\,\,\,\, \alpha=\frac{2-m}{2(1-m)} \in (1,3/2)$$ and 
$$ \eta=3/2-\alpha=(3(1-m)-(2-m))/(2(1-m))=(1-2m)/2(1-m).
$$
Therefore,
\bq
-\theta(r)&\ge&  \frac{2}{3} \frac{\tau^{3/2}}{(C_0+(1-m)\tau)^\alpha}\\
&=& \frac{2}{3}\left(  \frac{\tau^{3/2}}{(C_0+(1-m)\tau)^\alpha}\ind{0<\tau\le1} +   \frac{\tau^{3/2}}{(C_0+(1-m)\tau)^\alpha}\ind{\tau>1}\right)\\
&\ge& \frac{2}{3}\left(  \frac{\tau^{3/2}}{(C_0+(1-m))^\alpha}\ind{0<\tau\le1} +   \frac{\tau^{3/2}}{(C_0\tau+(1-m)\tau)^\alpha}\ind{\tau>1}\right)\\
&\ge& \frac{2}{3} \frac{1}{(C_0+(1-m))^\alpha} \min(\tau^{3/2},\tau^\eta).
\eq
\end{proof} 

\begin{remark}[On constants]\label{borneC1}{\rm
For reasons that will appear later during the study of our global optimization method, it is useful to have a compact expression for the inverse of the constant  $2/(3(C_0+(1-m))^\alpha)$ appearing in Lemma \ref{secondstep}. Using the equality $
\psi^\prime(\varphi'(\mu_{\min}))=1/\varphi^{\prime\prime}(\mu_{\min})$, 
this inverse writes
$$
C_1(\mu_{\min}):=\f32\max \left[(1+(1-m)(1-\varphi'(\mu_{\min})))^{{\frac{2-m}{2(1-m)}}},\sqrt{\varphi^{\prime\prime}(\mu_{\min})} \right]>1.
$$
Observe that, as a function of $\mu_{\min}$, $C_1$ is decreasing and therefore is bounded above by
\bq
\f32\max\Big\{ \left[1+(1-m)\left[1-\varphi'\left(\left[1+(1-m)\beta \osc(U)\right]^{1/(m-1)}\right)\right]\rt]^{\frac{2-m}{2(1-m)}}\,;\qquad&&
\\
 \sqrt{\varphi^{\prime\prime}((1+(1-m)\beta\osc(U))^{1/(m-1)}})\Big\}, && 
\eq
according to Proposition \ref{controlwedge}. When $\beta$ goes to $+\infty$, this bound behaves as $\displaystyle O\left(\beta^{\frac{2-m}{2(1-m)}}\right)$. Finally note that
\begin{equation}\label{c1}
C_1(\mu_{\min})|\theta(r)|\ge \min(|r|^{3/2},|r|^\eta)\mbox{ for all $r$.} 
\end{equation}

}
\end{remark}
\vspace{0.5cm}

\smallskip

We now turn to the desired lower bound for the squared Monge-Kantorovich  gradient $\J[\rho]$.
\begin{lem}[Lower bound for the squared Monge-Kantorovich  gradient]
We have
\bq
\J[\rho]=\int_M\vert \na \varphi'(\rho)-\na\varphi'(\mu)\vert^2\rho\, d\ell
&\geq & \int_M \vert \na \theta(g)\vert^2\, d\ell.
 \eq
\end{lem}
\begin{proof} 
Taking into account that both $\varphi'$ and $\psi$ are nondecreasing  functions, we get
\bq
\rho=\psi(g+\varphi'(\mu))&\geq &\psi( g+\varphi'(\mu_{\min}))\\
&=& (\theta'(g))^2.\eq
It ensues 
\bq\int_M\vert \na \varphi'(\rho)-\na\varphi'(\mu)\vert^2\rho\, d\ell
&=&\int_M\vert \na g\vert^2\rho\, d\ell\\
&\geq &\int_M\vert \na g\vert^2(\theta'(g))^2\, d\ell\\
&=&\int_M\vert \theta'(g)\na g\vert^2\, d\ell\\
&=&\int_M \vert \na \theta(g)\vert^2\, d\ell\eq
\end{proof} 
   
\subsection{Proof of Theorem \ref{funcineq} and its corollary}   
 
 Assume for the moment that $M$ is arbitrary.
 
 \smallskip
 \noindent
 In the previous section, we have proved two inequalities:
 $$
 \I[\rho]\le \int_M g^2 \,d\ell
 \mbox{ and }
\J[\rho]\ge \int_M \vert \na \theta(g) \vert^2\, d\ell.
$$
To reach a conclusion, it suffices to relate the quantities  $$\int_M g^2 \,d\ell \mbox{ and } \int_M \vert \na \theta(g) \vert^2\, d\ell.$$

\medskip

\noindent
Since $\eta\in(0,1/2)$,  Lemma \ref{secondstep} and \eqref{c1} (see Remark \ref{borneC1}) gives 
$$ 
C_1(\mu_{\min})|\theta(r)|\geq|r|^{3/2},\mbox{ when $|r|<1$ and }C_1(\mu_{\min})|\theta(r)|\geq |r|^\eta\mbox{ otherwise.}
$$ When $\beta$ is large enough $C_1(\mu_{\min})\geq 1$, we may thus write 
$$
g^2 \le C_2(\mu_{\min})\max(|h|^{\f43},|h|^{\f2\eta}),
$$
where we have set $h:=\theta(g)$ and $C_2=C^{\f2\eta}_1(\mu_{\min})$. Whence, taking the supremum and integrating yields
$$
\int_M g^2\,d\ell \le C_2(\mu_{\min}) \max(||h||_\iy^{\f43},||h||_\iy^{\f2\eta}).
$$

Recall that the function $\Omega:\R_+\to\R_+$ is such that
 \bq
\Omega(r)&\df&
\lt\{\begin{array}{ll}
r^{\f32}&\hbox{ if $r\in[0,1)$}\\[2mm]
 r^\eta&\hbox{ if $r\geq 1$.}
 \end{array} \rt.\eq
\noindent
Let us prove that the increasing function $\Omega$ satisfies the inequality
\bqn{propOmega}
\forall x>1,\,\forall y >0,\quad \Omega(xy) \le x^{\frac{3}{2}}\Omega(y).
\eqn
Let us consider two cases:\\
{\it First case: $y\ge 1$}. Because $x>1$, this implies $xy>1$. Thus,
$$
\Omega(xy)=(xy)^\eta=x^\eta\Omega(y)\le x^{\f32}\Omega(y).
$$
{\it Second case: $y < 1$}. When $xy>1$, the inequality follows as above. 
On the other hand, if $xy<1$, we have
$$
\Omega(xy)=(xy)^{\f32}=x^{\f32}\Omega(y).
$$
Therefore, by using the fact that $\Omega$ is increasing and satisfies \eqref{propOmega}, we get
$$
\Omega(\I[\rho])\le \Omega\left(\int_M g^2\,d\ell \right)\le (C_2(\mu_{\min}))^{\f32}||h||^2_\iy.
$$
To end the proof, it remains to compare
$$
||h||^2_\iy \hbox{ and }\int_M \vert \na h \vert^2\, d\ell.
$$

\smallskip
It is only at this point that we use the assumption about the dimension of $M$. 
\smallskip

Let us start first by an observation regarding regularity and prove that if $\J$ is finite then $\rho$ must be continuous. Observe first that for $\gamma>0$ and any measurable $\rho$, we have

\begin{eqnarray}
 \J(\rho)& = & \int_M\vert \na \varphi'(\rho)-\na\varphi'(\mu)\vert^2\rho\, d\ell\\
 & = & \int_M\vert \na \varphi'(\rho)\vert^2 \rho \,d\ell+ \int_M\vert \na \varphi'(\mu)\vert^2\,d\ell - 2\int_M\na \varphi'(\mu)\na\varphi'(\rho)\rho\, d\ell\\
 &\geq & \frac12\int_M\vert \na \varphi'(\rho)\vert^2 \rho\,d\ell - \int_M\vert \na \varphi'(\mu)\vert^2\rho\,d\ell 
  \end{eqnarray}
where we have used $2\vert \na \varphi'(\mu) \na\varphi'(\mu) \vert \leq 2\vert \na \varphi'(\mu)\vert^2 +\f12 \vert \na \varphi'(\rho)\vert^2$. Setting 
$$v(r)=\int_1^r\sqrt{s}\varphi'(s)\,ds, \mbox{ for $r>0$}$$
we see that $$\int_M\vert \na \varphi'(\rho)\vert^2 \rho\,d\ell=\int_M |\nabla v(\rho)|^2\,d\ell.$$  Thus the finiteness of $\J[\rho]$  implies that $\int_M |\nabla v(\rho)|^2\,d\ell$ is finite too.
\smallskip

Assuming now that $\dim M=1$,  standard results ensures that $v(\rho)$ is absolutely continuous and thus so is $\rho$
(so that we have furthermore $\int (v(\rho))^2\, d\ell<+\iy$ and $v(\rho)$  belongs to the Sobolev space $W^{1,2}(M)$).

Observe that we also obtain that $\rho$ is positive everywhere.  Since the function $\rho-\mu$ is continuous and satisfies
$\int_M \rho-\mu\, d\ell =0$, there exists at least one point $x_0$ in $M$, such that $\rho(x_0)-\mu(x_0)=0$.
It follows from \eqref{g} that $g(x_0)=0$ and from \eqref{theta} that $h(x_0)=0$ (where $h=\theta(g)$). For any $x\in \RR/(L\ZZ)$, denote by $[x_0,x]$ the shortest segment in $\RR/(L\ZZ)$ with boundary points $\{x_0,x\}$. Since $h$ is absolutely continuous:
\bq
h^2(x)&=&\lt(\int_{[x_0,x]}h'(y)\, \ell(dy)\rt)^2\\
&\leq & \ell([x_0,x]) \int_{[x_0,x]}(h'(y))^2\, \ell(dy)\\
&\leq & \f{L}2\int_M \vert \na h\vert^2\, d\ell.
\eq
Gathering the previous results gives
$$
\J[\rho] \ge  c(\beta) \Omega(\I[\rho]),
$$
with
 \begin{equation}\label{constant}
 c(\beta)=\f2L C_2(\mu_{\min})^{-\f{3}2}
  \end{equation}
  Since $ 2/\eta=4(1-m)/(1-2m)$ and for  $\beta$ large enough, $C_1=\displaystyle O\left(\beta^{\frac{2-m}{2(1-m)}}\right)$, one has $C_2=\displaystyle O\left(\beta^{\frac{2(2-m)}{1-2m}}\right)$ and
  $$c(\beta)=\displaystyle O\left(\beta^{\frac{-3(2-m)}{1-2m}}\right).$$

Let us conclude with the proof of Corollary~\ref{c:tal}. The key is to observe that $1/\sqrt{\Omega}$ is integrable at $0$, so that the inequality is a  \L ojasiewicz gradient inequality. It suffices to use the generalization of Otto-Villani theorem provided in \cite[Theorem 1(i)]{blanchet18}.

\section{Time-dependent swarm gradient methods}
Time dependence is key to obtain convergence to the actual global minimum: the penalty schedule $\beta$ is tuned so that the exploratory forces embodied in $\HH$ are sufficiently active at the beginning of the process while progressively loosing their influence on the dynamics as global goals have been achieved. In this second phase, as the diffusion process generated by $\HH$ fades away, the gradient dynamics of $U$ dominates and somehow terminates the process. As in the famous simulated annealing method, the presence of a functional inequality is fundamental for the dynamical system to converge.

\medskip

\noindent
In the remainder, unless otherwise stated we use a general potential $\varphi$.

\subsection{Main convergence results}
\paragraph{Convergence under a functional inequality}

The following general theorem shows the global optimization properties of \eqref{evol0}, provided that a functional inequality is available (as in simulated annealing or as in Theorem \ref{t:verrou} below).

\begin{theorem}[Global optimization under a weak functional inequality]\label{t:conv} Assume the set of hypothesis A, B are met,  and that $\beta,U,$ satisfy a  functional inequality of the type:
\begin{equation*}
\int_M\vert \na \varphi'(\rho)-\na\varphi'(\mu)\vert^2\rho\, d\ell\geq  c(\beta)\,\Omega\lt(\int_M\varphi(\rho)-\varphi(\mu)-\varphi'(\mu)(\rho-\mu)\, d\ell\rt)
\end{equation*}
where $c,\Omega:(0,+\infty)\to\R$ are positive with $\Omega$ being nondecreasing. If the penalization schedule $t\mapsto\beta(t)$ satisfies 
\begin{align}
\lim_{t\ri+\iy} \dot{\beta}(t)/c(\beta (t))=0 \label{1}\\
\int_1^{+\iy}c(\beta(t))\, dt =+\iy\label{2}
\end{align}
then
\begin{equation}
\U[\rho(t)]-\min_M U= \int_M U\rho(t)-\min_M U\: \leq \: \app(\beta_t) +o( \beta_t^{-1}),
\end{equation}
where the quantity $\app(\beta_t)\to 0$ was defined in \eqref{app}.
\end{theorem}

\begin{remark}[Existence of a schedule]\label{r:schedule}{\rm Assume that the function $c$ satisfies for large $\beta>0$,
\bq
c(\beta)&=& O(\beta^{-\gamma})\eq
for some exponent $\gamma>0$ (as it is the case in Proposition \ref{funcineq}).
Then for any $\alpha\in(0,1/(1+\gamma))$,  any penalization schedule $t\mapsto \beta(t)$ 
such that for large enough $t>0$, $\beta(t)= t^{\alpha}$
satisfies the assumptions of the convergence theorem, as it is readily checked.}
\end{remark}

\paragraph{Proof of Theorem \ref{t:conv}}
\begin{proof}
Recall that the evolution equation \eqref{evol0} writes
$$
\fo t\geq 0,\qquad\dt  \rho(t)=\dive(\rho(\beta_t\na U+\na \varphi'(\rho))).
$$
In view of \eqref{rhopsicU}, the curve of stationary measures $\mut\df\mu_{\beta_t}$, satisfies, for each fixed $t$
\begin{equation}\label{opt}
\beta_t\na U+\na \varphi'(\mut )\;=\;0.
\end{equation}
The functionals $\I$ and $\J$ used in Section~\ref{s:globdyn} are adapted to the time-inhomogeneous case as follows
$$
\I[t,\rho]\df\int_M\varphi(\rho)\, d\ell+\int_M \beta_t U\,\rho \,d\ell -\lt(\int_M\varphi(\mut )\, d\ell+\int_M \beta_t U\,\mut \,d\ell\rt),
$$
and
$$
\J[t,\rho]=\int_M\vert \na \varphi'(\rho)-\na\varphi'(\mut )\vert^2\rho\, d\ell,
$$
for any admissible $\rho$ and time $t\geq0$. 
Using regularity results \eqref{ac0}-\eqref{ac1}, the differentiation of the objective along the evolution curve $t\mapsto\rho$ of \eqref{evol0} yields
\bq
\dt  \I[t,\rho(t)]&=&\int\varphi'(\rho)\dt {\rho}\, d\ell+\int \beta_t U\dt {\rho}\, d\ell+ \dt {\beta_t}\int_M U(\rho-\mut )\,d \ell  \\
&=&\int(\varphi'(\rho)+\beta_tU)\dive(\rho(\beta_t \na U+\na \varphi'(\rho)))\, d\ell + \dot{\beta_t}\int_M U(\rho-\mut )\,d \ell \\
&=&-\int\na (\varphi'(\rho)+\beta_t U)(\beta_t \na U+\na \varphi'(\rho)))\, \rho d\ell + \dot{\beta_t}\int_M U(\rho-\mut )\,d \ell \\
&=&-\int_M\vert \na \varphi'(\rho)+\beta_t\na U\vert^2\, \rho  d \ell+ \dot{\beta_t}\int_M U(\rho-\mut )\,d \ell\\
&=&-\J[t,\rho(t)]+ \dot{\beta_t}\int_M U(\rho-\mut )\,d \ell,
\eq
where we have used \eqref{opt} in the first equality, the evolution equation in the second-one, and finally integration by parts.  
Let us set $v(t)=\I(t,\rho)$. Using the functional inequality gives
$$
\dot v(t) \le -c(\beta_t) \Omega( v(t))+\dot{\beta_t}\int_M U(\rho-\mut )\,d \ell.
$$
To give an upper bound of the second term, we write
$$
\lve\dot{\beta_t}\int_M U(\rho-\mut )\,d \ell\rve=\lve\dot{\beta_t}\left(\int_M U\rho\,d \ell-\int_MU\mut \,d \ell\right)\rve \le \osc(U)\lve\dot{\beta_t}\rve, 
$$
where the last inequality uses that $\rho$ and $\mut $ are probability densities on $M$. 
Finally, we obtain
$$
\dot v(t) \le -c(\beta_t) \Omega( v(t))+\osc(U)\lve\dot{\beta_t}\rve,
$$

\medskip

We are thus led to consider  differential inequalities of the type
\bq\label{}
\dot{v}&\leq & -c(\beta)\Omega(v)+\delta\lve\dot{\beta}\rve\eq
where $v\st \RR_+\ri \RR_+$ is a non-negative function, $d>0$ is a  positive constant and $\Omega$ is an nondecreasing  function
taking positive values on $(0,+\iy)$.

\smallskip

Our goal is now to give conditions on the inverse temperature scheme $\beta\st \RR_+\ri(0,+\iy)$ ensuring that $v$ converges to zero for large times:
\begin{pro}[Schedule conditions]\label{pro}
Assume that for large times, \eqref{1} and \eqref{2} hold, 
then
\bq
\lim_{t\ri+\iy} v(t)&=&0.\eq
\end{pro}
\smallskip
\noindent
The proof of Proposition \ref{pro} is based on the two following observations.
\begin{lem}\label{liminfv}
Under the assumptions of Proposition \ref{pro}, we have
\bq
\liminf_{t\ri+\iy} v(t)&=&0.\eq
\end{lem}
\begin{proof} 
Towards a contradiction assume that there exist $\epsilon>0$ and $T_0\geq 0$ such that
$v(t)\geq \epsilon$ for all $t\geq T_0$.
Then, $\Omega$ being nondecreasing,
\bq
\fo t\geq T_0,\qquad \dot{v}(t)&\leq & -c(\beta(t))\Omega(\epsilon)+\delta\lve\dot{\beta}(t)\rve.
\eq
From the first condition \eqref{1}, there exists $T_1\geq T_0$ such that 
\bq
\fo t\geq T_1,\qquad \lve\dot{\beta}(t)\rve&\leq &\f{1}{2}c(\beta(t))\Omega(\epsilon)\eq
and thus
\bq
\fo t\geq T_1,\qquad \dot{v}(t)&\leq & -\f{1}{2}c(\beta(t))\Omega(\epsilon).
\eq
This implies
\bq
\fo t\geq T_1,\qquad
v(t)&\leq &v(T_1) -\f{1}{2}\Omega(\epsilon)\int_{T_1}^tc(\beta(s))\, ds,\eq
so letting $t$ go to infinity, due to the second condition \eqref{2}, $\lim_{t\ri+\iy} v(t)\,=\,-\iy$,   in contradiction with the non-negativity of $v$.
\end{proof} 
\par
Our second ingredient is:
\begin{lem}\label{downward}
Under the assumptions of Proposition \ref{pro}, fix $\epsilon>0$. Then there exists
$T\geq 0$, such that 
\bq
\fo t\geq T,\quad \big( v(t)=\epsilon&\Rightarrow& \dot{v}(t)<0\big).
\eq
\end{lem}
\begin{proof} 
Indeed, consider \bq
T&\df& \inf\lt\{\tau\geq 0\st \fo t\geq \tau,\ -\f{1}{2}c(\beta(t))\Omega(\epsilon)+\delta\lve\dot{\beta}(t)\rve\leq 0\rt\}\eq
which is finite by assumption \eqref{1}. 
Proceeding as in  Lemma~\ref{liminfv}, for any $t\geq T$ such that $v(t)=\epsilon$, we may assert that
\bq
\dot{v}(t)&\leq &-\f{1}{2}c(\beta(t))\Omega(\epsilon)\,<\,0.\eq
\end{proof} \par
\sm
The proof of Proposition \ref{pro} is as follows:
by Lemma~\ref{liminfv}, for any arbitrary small level $\epsilon>0$, 
$v$ will always go below $\epsilon$ at some point, but by Lemma~\ref{downward} there exists a time $T$ after which it will no longer be able to cross it upward. Whence $v$ must be below $\epsilon$ for large times. This concludes the proof of Proposition~\ref{pro}.

\smallskip

Let us now come back to the proof of Theorem~\ref{t:conv}: since $v(t)$ tends to zero, we have for all $\epsilon>0$ a $T_0$ such that 
$$
\I[t,\rho]\df\int_M\varphi(\rho)\, d\ell+\int_M \beta_t U\,\rho \,d\ell -\lt(\int_M\varphi(\mut )\, d\ell+\int_M \beta_t U\,\mut \,d\ell\rt)\leq \epsilon
$$
that is 
\begin{eqnarray*}
\int_M U\,\rho \,d\ell -\int_M U\,\mut \,d\ell & \leq & \beta_t^{-1} \left( \int_M\varphi(\mut )\, d\ell-\int_M\varphi(\rho)\, d\ell +\epsilon\right)
\end{eqnarray*}
which implies in turn 
\begin{eqnarray*}
\int_M U\,\rho \,d\ell -\min_M U & \leq & \beta_t^{-1} \left( \int_M\varphi(\mut )\, d\ell+\epsilon\right) +\left(\int_M U\,\mut \,d\ell-\min_M U\right)\\
& \leq & \app(\beta_t) +\epsilon \beta_t^{-1}.
\end{eqnarray*}
The value of $\epsilon$ being arbitrary, the results follows.\end{proof}

\paragraph{One-dimensional global optimization}

We may now combine Theorem~\ref{t:conv} with the functional inequality we obtained previously in Theorem~\ref{funcineq}.

\begin{theorem}[Global optimization by swarm gradient in dimension 1]\label{t:verrou}
Assume that $M$ is  the circle $\TT\df\RR/(L\ZZ)$ endowed with its usual Riemannian structure, that the function $\varphi$ is as in \eqref{varphi}, i.e., $\varphi=\varphim$ and that the schedule is given by
$$\beta(t)=kt^{1/\gamma}, \mbox{ with $k>0$ and $\gamma=\frac{3(2-m)}{1-2m}\in [6,+\infty)$.}$$

Then 
\begin{equation}
\lim_{t\to\infty}\U[\rho(t)]=\min_M U
\end{equation}
\end{theorem}

\begin{proof}
Theorem~\ref{funcineq} provides a functional inequality as required by Theorem~\ref{t:conv}.  From this inequality we may assume that $c(\beta)=\kappa \beta^{-\gamma}$ for some $\kappa>0$.
Now, using Remark~\ref{r:schedule}, we choose $\beta(t)=c^{-1}(t)=\kappa^{-1}t^{\frac{1}{\gamma}}$. This is the choice made by assumption provided that $k=\kappa^{-1}$. One easily checks that the assumptions of Theorem~\ref{t:conv} are satisfied, hence the result.
\end{proof}

\subsection{A stochastic view on the dynamics: particles swarm optimization}\label{s:swarm}
In this last section, we sketch a possible  discrete approach of the  swarm gradient dynamics; it will be properly developed in future works.
 It relies on the diffusion process associated with the evolution equation
\bqn{evol}
\fo t\geq 0,\qquad\dt \rho &=&\dive(\rho (\beta_t\na U+\na \varphi'(\rho))).
\eqn 
To evidence this link, let us use the formal integration by parts presented in Section \ref{s2.2} to obtain  at any time $t\geq 0$ and for every $ f\in\cC^\iy(M)$, 
\bq
\int_{\{\rho_t>0\}} L_{t,\rho}[f]\, \rho_t d\ell&=&\int_M f \dive(\rho(\beta_t\na U+\na \varphi'(\rho)))\, d\ell
\eq
where
\bqn{generator}
L_{t,\rho}[f]&=&\alpha(\rho)\tr f-\lan\beta_t \na U,\na f\ran
\eqn
with $\alpha\st(0,+\iy)\ri\RR_+$ given by
\bq
\fo r> 0,\qquad \alpha(r)&\df &\f1r\int_0^r s\varphi''(s)\, ds.
\eq
\par
It is then natural to associate to equation \eqref{evol} a Markov process $(X_t)_{t\geq 0}$ whose infinitesimal generator is given by \eqref{generator} and whose law has density $\rho(t)$ for all $t\geq0$. 
Due to the dependence of the evolution on the density, such a Markov process is said to be non-linear.
When $M$ is a flat torus of dimension $d$, 
 consider the stochastic differential equation
\bqn{SDE}
X_t=X_0-\int_0^t \beta_s\na U(X_s)\,ds+\int_0^t \sqrt{\alpha(\rho(X_s))} \,dB_s,
\eqn
where $(B_s)_{s\geq 0}$ is a Brownian motion of dimension $d$. If $(X,\rho)$ is a solution of the stochastic differential equation \eqref{SDE} then an application of the It\^o formula shows that $\rho$ is a solution of equation \eqref{evol}.
This observation can be extended to any compact Riemannian manifold.\par
Two particular choices for which the existence of $(X,\rho)$  has been established in the literature are
\begin{itemize}
\item[(i)] $\varphi(\rho)=\rho\log(\rho)$, we have $\alpha\equiv 1$, so $L_{t,\rho}$ does not depend on $\rho$ and is  the Langevin generator $\tr\cdot -\beta_t\lan \na U,\na \cdot\ran$. The existence and uniqueness of a strong solution to the extension of \eqref{SDE} on a compact Riemannian is well-known,  as soon as $\na U$ is Lipschitz, which is true if $U$ is smooth.
\item[(ii)] $\varphi(\rho)=\rho^m$ which corresponds to the equation of porous media. 
On $\RR$, the existence and uniqueness of a strong solution to \eqref{SDE} is proven by Benachour,  Chassaing,  Roynette and Vallois \cite{ASNSP_1996_4_23_4_793_0}
(see {Belaribi} and  {Russo} \cite{zbMATH06098167} for more general functions $\varphi$).
\end{itemize}
\par
The main difficulty in investigating the existence and uniqueness of \eqref{SDE} is the presence of the density $\rho_s$. It can be relaxed if it could be replaced by integrals of smooth functions with respect to $\rho_s(x)\, \ell(dx)$.
It leads us to consider convolutions of $\rho_s$ with respect to some smooth kernels, as those traditionally used in statistics for density estimation, cf.\ e.g.\ Silverman \cite{zbMATH04001209}. In our geometric setting, it seems natural to resort to the heat kernel (associated to the Laplace-Beltrami operator).
To avoid the introduction of further notation, let us just consider the case of a flat torus $M=(\RR/\ZZ)^d$, endowed with its usual Riemannian structure.
Let $K\st (\RR/\ZZ)^d\ri \RR_+$ be a smooth function with support in $[-1/4,1/4]^d$ (seen as a subset of $(\RR/\ZZ)^d$) and such that $\int_{M}K\, d\ell =1$.
For any probability density $\rho$ on $M$ and $h\in(0,1)$, which is a bandwidth parameter, set
\bqn{rhoh}
\fo x\in M,\qquad \rho_h(x)&\df& h^{-d}\int_M K((x-y)/h)\, \rho(y) \ell(dy).\eqn
\par
When $h$ is small, $\rho_h$ is an approximation of $\rho$.
Replacing in \eqref{SDE}, $\rho_s$ by $\rho_{h,s}$, we end up with the stochastic differential equation 
\bqn{SDE2}
X_t=X_0-\int_0^t \beta_s\na U(X_s)\,ds+\int_0^t \sqrt{\alpha(\rho_{h}(X_s))} \,dB_s,
\eqn
which is simpler to investigate, taking into account the usual mean field theory (see for instance Del Moral \cite{MR3060209}).
\par
Furthermore, \eqref{SDE} admits natural particles approximations.
Indeed, \eqref{rhoh}   can be extended to any probability measure $\pi$ on $M$, by replacing $\rho(y) \ell(dy)$ by $\pi(dy)$. 
In particular, it makes sense for the empirical measure of $N$ particles, which is a crucial feature for particle approximations.
More precisely,  consider
a system of $N$ particles, $X_1, X_2, ..., X_N$ whose joint evolution is described by the stochastic differential equations,
\bqn{Xn}
\fo n\in\lin N\rin,\qquad dX_n(t)&=&-\beta_t\na U(X_n(t))+\sqrt{\alpha(\rho_{N,h}(X_n(t)))} \,dB_n(t)\eqn
where
the $(B_n(t))_{t\geq 0}$, for $n\in\lin N\rin$, are independent Brownian motions of dimension $d$,
and where
\bq
\fo x\in M,\qquad \rho_{N,h}(x)&\df& \frac{1}{N}\sum_{n=1}^N \frac{1}{h}K\left( \frac{x-X_n(t)}{h}\right)\eq
namely $\rho_{N,h}$ is given by \eqref{rhoh} when $\rho$ is replaced by the empirical measure
\bq
\rho_{N}(t)&\df& \f1N\sum_{n=1}^N\delta_{X_n(t)}\eq
\par
Resorting to mean field theory and chaos propagation, when $N$ tends to infinity, the random evolution $(\rho_{N}(t))_{t\geq 0}$ converges in probability toward the dynamical system 
$(\rho(t))_{t\geq 0}$ (for related initial conditions), the law of $(X_1(t))_{t\geq 0}$ converges toward that of $(X_t)_{t\geq 0}$, solution of \eqref{SDE2}, and the trajectories of any fixed finite number of particles become asymptotically independent.
\par
Putting together these observations, we believe that in addition to $(\beta_t)_{t\geq 0}$, schemes $ (h_t)_{t\geq 0}$ and $(N_t)_{t\geq 0}$ can be found, with $\lim_{t\ri+\iy} h_t=0$ and 
$\lim_{t\ri+\iy} N_t=\iy$, so that for large times $t$, the corresponding empirical measures $\rho_{N_t}$ concentrate around the set of global minima of $U$.
There are several ways to increase the number of particles,  the most natural one might be to duplicate some of the current particles --but it is also possible to make them appear in independent random positions.\par
This procedure would provide a new stochastic algorithm for global minimization. At least up to the simulation of particle systems such as \eqref{Xn}, but this can be done through traditional Euler-Maruyama scheme.
\par

\paragraph{Acknowledgements.} 

JB acknowledges the support of ANR-3IA Artificial and Natural Intelligence Toulouse Institute, and thank Air Force Office of Scientific Research, Air Force Material Command, USAF, under grant numbers FA9550-19-1-7026, and ANR MasDol. JB, LM, SV acknowledges the support of ANR Chess, grant ANR-17-EURE-0010 and TSE-P.

\bibliographystyle{siam}

\end{document}